\documentclass[reqno,12pt]{amsart}
\usepackage{bbm}
\usepackage[all,pdf]{xy}
\usepackage{epsfig}
\usepackage{amsmath}
\usepackage{amssymb}
\usepackage{amscd}
\usepackage{graphicx}

\allowdisplaybreaks[4]

\makeatletter
\@namedef{subjclassname@2020}{%
	\textup{2020} Mathematics Subject Classification}
\makeatother
\usepackage[colorlinks=true]{hyperref}
\usepackage{amsfonts}
\usepackage[top=25mm, bottom=27mm, left=27mm, right=27mm]{geometry}

\newtheorem{thm}{Theorem}[section]
\newtheorem{lemma}[thm]{Lemma}

\newtheorem{prop}[thm]{Proposition}

\newtheorem{cor}[thm]{Corollary}
\newtheorem{defn}[thm]{Definition}
\newtheorem{rem}[thm]{Remark}
\newtheorem{step}{Step}

\newcommand{\norm}[1]{\left\Vert #1\right\Vert}
\newcommand{\nnorm}[1]{\lvert\!|\!| #1|\!|\!\rvert}

\def \N {\mathbb N}
\def \C {\mathbb C}
\def \Z {\mathbb Z}
\def \R {\mathbb R}

\def\B {\mathcal B}

\numberwithin{equation}{section}

\begin{document}
		
	\baselineskip 14pt
	
	\title[]{Multilinear Wiener-Wintner type ergodic averages and its application}

	\author[]{Rongzhong Xiao}
	
	\address{School of Mathematical Sciences, University of Science and Technology of China, Hefei, Anhui, 230026, PR China}
	\email{xiaorz@mail.ustc.edu.cn}

\subjclass[2020]{Primary: 37A30; Secondary: 37A46.}
\keywords{multilinear Wiener-Wintner type ergodic averages, nilsequence, cubic averages, polynomial ergodic averages, Furstenberg systems.}

 \begin{abstract}
 	In this paper, we extend the generalized Wiener–Wintner Theorem built by Host and Kra to the multilinear case under the hypothesis of pointwise convergence of multilinear ergodic averages. In particular, we have the following result:
 	
 	Let $(X,\B,\mu,T)$ be a measure preserving system. Let $a$ and $b$ be two distinct non-zero integers. Then for any $f_{1},f_{2}\in L^{\infty}(\mu)$, there exists a full measure subset $X(f_{1},f_{2})$ of $X$ such that for any $x\in X(f_{1},f_{2})$, and any nilsequence $\textbf{b}=\{b_n\}_{n\in \Z}$, $$ \lim_{N\rightarrow \infty}\frac{1}{N}\sum_{n=0}^{N-1}b_{n}f_{1}(T^{an}x)f_{2}(T^{bn}x)$$ exists.
	\end{abstract}
		\maketitle
		
	\section{Introduction}
		Throughout the paper, by a \emph{measure preserving system} or a \emph{system}, we mean a Lebesgue space $(X,\B,\mu)$ with an invertible measure preserving transformation $T:X\rightarrow X$. We write a system as $(X,\B,\mu,T)$. A system $(X,\B,\mu,T)$ is ergodic if the only $T$-invariant subsets in $\B$ have measure $0$ or $1$.
		
		In 1941, Wiener and Wintner strengthened the classical Birkhoff's pointwise ergodic theorem.
		\begin{thm}
			$($\cite{WW}$)$For a system $(X,\B,\mu,T)$, and any $f\in L^{\infty}(\mu)$, there exists a full measure subset $X(f)$ of $X$ such that for any $x\in X(f)$, and any $t\in \R$, the limit $$\lim_{N\rightarrow \infty}\frac{1}{N}\sum_{n=0}^{N-1}e^{2\pi int}f(T^{n}x)$$ exists.
		\end{thm}
		
	   The core of the Wiener–Wintner Theorem is that $X(f)$ is independent of the choice of $t$. Since $\R$ is uncountable, the result leaded reseachers to focus on general phenomena in ergodic theory where samplings are good for an uncountable number of systems. And for the convergent behavior of some polynomial ergodic averages, the Wiener–Wintner Theorem can also help us to understand it.
		
		In 2009, Host and Kra built the following result, which is called generalized Wiener–Wintner Theorem.
		\begin{thm}\label{thm3}
			$($\cite[Theorem 2.22]{HKU}$)$For a system $(X,\B,\mu,T)$, and any $f\in L^{\infty}(\mu)$, there exists a full measure subset $X(f)$ of $X$ such that for any $x\in X(f)$, and any nilsequence $($for definition, see Subsection \ref{subsec}.$)$ $\textbf{b}=\{b_n\}_{n\in \Z}$, the limit $$\lim_{N\rightarrow \infty}\frac{1}{N}\sum_{n=0}^{N-1}b_{n}f(T^{n}x)$$ exists.
		\end{thm}
	   Note that for any $t\in\R$, $\{e^{2\pi int}\}_{n\in \Z}$ is a nilsequence since $\R/\Z$ is a $1$-step nilmanifold. So the above statements generalize Wiener-Wintner theorem. In addition, Host and Kra pointed out the following result in \cite{HK-book}.
	   \begin{thm}\label{P1}
	   	$($\cite[Theorem 23.5]{HK-book}$)$Given $d\in \N$, let $a_1,\cdots,a_d$ be distinct non-zero integers. Let $(X,\B,\mu,T)$ be a measure preserving system. Then for any $f_{1},\cdots,f_{d}\in L^{\infty}(\mu)$ and any nilsequence $\textbf{b}=\{b_n\}_{n\in \Z}$, $$ \lim_{N\rightarrow \infty}\frac{1}{N}\sum_{n=0}^{N-1}b_{n}\prod_{j=1}^{d}f_{j}(T^{a_{j}n}x)$$ exists in $L^{2}(\mu)$. 
	   \end{thm}
	   The averages which are mentioned in the above theorem are called multilinear Wiener-Wintner type ergodic averages. The above theorem guarantees the norm convergence of multilinear Wiener-Wintner type ergodic averages. For extra results related to the above theorem, one can see  \cite[Theorem B.4]{HSY}.
	   
	   Based on Theorem \ref{thm3} and Theorem \ref{P1}, there is a natural question: can we generalize Theorem \ref{thm3} to multilinear version under the hypothesis of pointwise convergence of multilinear ergodic averages?
	   
	   For the bilinear case, one can find some results in \cite{AM,ADM}. For the general case, we obtain the following result.
	\begin{thm}\label{A}
		Given $d\in \N$, let $a_1,\cdots,a_d$ be distinct non-zero integers. Let $(X,\B,\mu,T)$ be a measure preserving system. Then for any $h_{1},\cdots,h_{d}\in L^{\infty}(\mu)$, $$\lim_{N\rightarrow \infty}\frac{1}{N}\sum_{n=0}^{N-1}\prod_{j=1}^{d}h_{j}(T^{a_{j}n}x)$$ exists almost everywhere if and only if for any $f_{1},\cdots,f_{d}\in L^{\infty}(\mu)$, there exists full measure subset $X(f_{1},\cdots,f_{d})$ of $X$ such that for any $x\in X(f_{1},\cdots,f_{d})$, and any nilsequence $\textbf{b}=\{b_n\}_{n\in \Z}$, $$ \lim_{N\rightarrow \infty}\frac{1}{N}\sum_{n=0}^{N-1}b_{n}\prod_{j=1}^{d}f_{j}(T^{a_{j}n}x)$$ exists. 
	\end{thm}
	
	In 2015, Zorin-Kranich \cite[Corollary 1.4]{ZK} had pointed out the above theorem for ergodic systems and provided an inductive proof via harmonic analsis's method. In this paper, the author will give a purely ergodic theoretic proof for the above theorem based on Host-Kra factor and determine related characteristic factor.
	
	Based on the above theorem, \cite[Theorem]{B}, \cite[Theorem C]{HSY19} and \cite[Theorem 2.24]{HKU}, we have the following two corollaries.
	\begin{cor}\label{cor1}
		Let $(X,\B,\mu,T)$ be a measure preserving system. Let $a$ and $b$ be two distinct non-zero integers. Then for any $f_{1},f_{2}\in L^{\infty}(\mu)$, there exists full measure subset $X(f_{1},f_{2})$ of $X$ such that for any $x\in X(f_{1},f_{2})$, and any nilsequence $\textbf{b}=\{b_n\}_{n\in \Z}$, $$ \lim_{N\rightarrow \infty}\frac{1}{N}\sum_{n=0}^{N-1}b_{n}f_{1}(T^{an}x)f_{2}(T^{bn}x)$$ exists.
	\end{cor}
	\begin{cor}\label{cor2}
		Given $d\in \N$, let $a_1,\cdots,a_d$ be distinct non-zero integers. Let $(X,\B,\mu,T)$ be an ergodic distal measure preserving system $($for definition, see \cite[Definition 5.4]{HSY19}$)$. Then for any $f_{1},\cdots,f_{d}\in L^{\infty}(\mu)$, there exists full measure subset $X(f_{1},\cdots,f_{d})$ of $X$ such that for any $x\in X(f_{1},\cdots,f_{d})$, we have
		\begin{itemize}
			\item[(1)] for any nilsequence $\textbf{b}=\{b_n\}_{n\in \Z}$, $$ \lim_{N\rightarrow \infty}\frac{1}{N}\sum_{n=0}^{N-1}b_{n}\prod_{j=1}^{d}f_{j}(T^{a_{j}n}x)$$ exists;
			\item[(2)] for any system $(Y,\mathcal{D},\nu,S)$, any $k\ge 1$, and any $g_{1},\cdots,g_{k}\in L^{\infty}(\nu)$,  $$\lim_{N\rightarrow \infty}\frac{1}{N}\sum_{n=0}^{N-1}\prod_{j=1}^{d}f_{j}(T^{a_{j}n}x)\prod_{i=1}^{k}S^{in}g_i(y)$$ exists in $L^{2}(\nu)$. 
		\end{itemize}
	\end{cor}
	We remark for the $(2)$ of Corollary \ref{cor2}, Assani and Moore \cite[Theorem 1.5]{AM2015} built the corresponding result for $d=2$ without requiring ergodic distal condition. 
	
		Next, we apply Theorem \ref{A} to study some polynomial ergodic averages for multiple measure preserving transformations under lacking complete commutativity. Due to lacking complete commutativity, we have to do some restrictions for transformations such as zero entropy which can help us to use disjointness between zero entropy and completely positive entropy to attack some polynomial ergodic averages. The reason for which Theorem \ref{A} can work for some polynomial ergodic averages is the characteristic behavior of $\infty$-step factor $($ for definition, see Remark \ref{rem}.$)$.
		
		Recently, Frantzikinakis and Host established a result on norm convergence of some polynomial ergodic averages. 
		\begin{thm}
		$($\cite[Theorem 1.1]{HN}$)$Let $T,S$ be invertible measure preserving transformations acting on a Lebesgue space $(X,\B,\mu)$ such that $(X,\B,\mu,T)$ has zero entropy and let $p(n)$ be a non-constant non-linear integer coefficents polynomial. Then for any $f,g\in L^{\infty}(\mu)$, the limit $$\lim_{N\rightarrow \infty}\frac{1}{N}\sum_{n=1}^{N}f(T^{n}x)g(S^{p(n)}x)$$ exists in $L^{2}(\mu)$.
		\end{thm}
		 The zero entropy assumption on $T$ in the above result is necessary. If there is no zero entropy restriction for $T$, one can find a counterexample  from \cite[Proof of Proposition 1.4]{HN}. Motivated by the above theorem, we establish the following result.
		\begin{thm}\label{B}
			Given $d\in \N$, let $T,S_{1},\cdots,S_{d}$ be invertible measure preserving transformations acting on a Lebesgue space $(X,\B,\mu)$ such that $(X,\B,\mu,T)$ has zero entropy and $S_{1},\cdots,S_{d}$ are commuting and let $p_{1}(n),\cdots,p_{d}(n)$ be non-constant non-linear integer coefficents polynomials with distinct degrees. Let $a$ and $b$ be two distinct non-zero integers. Then for any $f_{1},f_{2},g_{1},\cdots,g_{d}\in L^{\infty}(\mu)$, the limit $$\lim_{N\rightarrow \infty}\frac{1}{N}\sum_{n=1}^{N}f_{1}(T^{an}x)f_{2}(T^{bn}x)\prod_{j=1}^{d}g_{j}(S_{j}^{p_{j}(n)}x)$$ exists in $L^{2}(\mu)$.
		\end{thm}
	\subsection*{Organization of the paper}In Section 2, we recall some notions and needed results. In section 3, we introduce two useful estimations. In sections 4 and 5, we prove Theorem \ref{A} and Theorem \ref{B} respectively.

	\section{Preliminaries}
	In this section we recall some basic notions and results.
	\subsection{Factor and joining}
	A \emph{factor} of a system $(X,\B,\mu,T)$ is a $T$-invariant sub-$\sigma$-algebra of $\B$. A \emph{factor map} from $(X,\B,\mu,T)$ to $(Y,\mathcal{D},\nu,S)$ is a measurable map $\pi:X_{0}\rightarrow Y_{0}$ with $\pi\circ T=S\circ\pi$ and such that $\nu$ is the image of $\mu$ under $\pi$ where $X_0$ is a $T$-invariant full measurable subset of $X$ and $Y_0$ is a $S$-invariant full measurable subset of $Y$. In this case, $\pi^{-1}(\mathcal{D})$ is a factor of $(X,\B,\mu,T)$ and every factor of $(X,\B,\mu,T)$ can be obtained in this way.
	
	A \emph{joining} of two systems $(X,\B,\mu,T)$ and $(Y,\mathcal{D},\nu,S)$ is a probability measure $\lambda$ on $X\times Y$, invariant under $T\times S$ and whose projections on $X$ and $Y$ are $\mu$ and $\nu$ respectively. Likely, we can define joining on more systems. 
	
	To study non-conventional ergodic averages on $(X,\B,\mu,T)$, \emph{Furstenberg's self-joining} was introduced.
	
	At first, we claim some notations. Given $d\in \N$, let $X^{d}$ denote the space $X\times \cdots  \times X(d\ times)$ where an element of $X^{d}$ can be written as $\textbf{x}=(x_{1},\cdots,x_d)$ and $\B^{\otimes d}$ denote the product $\sigma$-algebra. 
	
	Let $\mathcal{A}=\{a_{1}n,\cdots,a_{d}n\}$ where $a_{1},\cdots,a_{d}$ are distinct non-zero integers. We define a measure $\mu^{\mathcal{A}}_{d}$ on $(X^{d},\B^{\otimes d})$ by 
	$$
	\mu^{\mathcal{A}}_{d}(A_{1}\times \cdots \times A_{d})=\lim_{N\rightarrow \infty}\frac{1}{N}\sum_{n=0}^{N-1}\int_{X}\prod_{i=1}^{d}T^{a_{i}n}1_{A_{i}}d\mu
	$$
	where $A_{i}\in \B,1\le i\le d$. Existence of the above limit can be guaranteed by \cite[Theorem 1.1]{HK}.
	We refer to $\mu^{\mathcal{A}}_{d}$ as the Furstenberg's self-joining of $(X,\B,\mu,T)$ with respect to polynomial family $\mathcal{A}$ in the light of its historical genesis in Furstenberg’s work on the ergodic-theoretic proof for Szemer{\'e}di’s theorem \cite{F}. By an easy observation, we have the following result.
	\begin{prop}\label{prop1}
		$(X^{d},\B^{\otimes d},\mu^{\mathcal{A}}_{d},T^{a_{1}}\times \cdots \times T^{a_{d}})$ is a measure preserving system.
	\end{prop}
	\subsection{Nilsystem and nilsequence}\label{subsec}
	Let $k\in\N$. A \emph{$k$-step nilmanifold} $X$ is the quotient space $G/\Gamma$ where $G$ is a $k$-step nilpotent Lie group and $\Gamma$ is a cocompact discrete subgroup of $G$. The group $G$ acts on $X$ by left translations. The unique invariant Borel probability measure on $X$ under the action is called the Harr measure of $X$ denoted by $m_{X}$. If $a\in G$ and $T_{a}:X\rightarrow X$ is the translation $x\mapsto a\cdot x$, the system $(X,\B_{X},T_{a},m_{X})$ is a \emph{$k$-step nilsystem} where $\B_{X}$ is the Borel $\sigma$-algebra of $X$.
	
	A \emph{basic $k$-step nilsequence} is the sequence $\{f(a^{n}\cdot x)\}_{n\in \Z}$ where $f\in C(X),a\in G,x\in X$. A \emph{$k$-step nilsequence} is a uniform limit of basic $k$-step nilsequences. Clearly, all $k$-step nilsequences is an invariant  algebra of $l^{\infty}(\Z)$ under translation.(For more details, see \cite[Chapter 11]{HK-book}).
	
	\subsection{Host-Kra seminorm and local seminorm}
	In 2005, Host and Kra introduced Host-Kra seminorm on $(X,\B,\mu,T)$ when they considered the norm convergence of non-conventional ergodic averages in \cite{HK}.
	
	At first, we introdue some notations. Given $k\in \N$, let $V_{k}=\{0,1\}^{k}$ where an element of $V_{k}$ can be written as $\underline{\epsilon}=(\epsilon_1,\cdots,\epsilon_k)$,  $X^{[k]}=X^{2^{k}}$ where an element of $X^{[k]}$ can be written as $\underline{x}=(x_{\underline{\epsilon}}: \underline{\epsilon}\in V_{k})$ ,$T^{[k]}=T\times \cdots \times T(2^{k}\ times)$ and $\B^{\otimes [k]}=\B^{\otimes 2^{k}}$. Given $N\in \N$, let $[N]^{k}=\{\underline{h}=(h_1,\cdots,h_k):0\le h_1,\cdots,h_k\le N-1\}$. 
	
	Now, we define $\mu^{[k]}$ by induction. Let $(\Omega,\mathcal{F},m,S) $ denote the factor associated to $\mathcal{I}(T)$ where $\mathcal{I}(T)$ is a sub-$\sigma$-algebra of $\B$ consisting of all $T$-invariant subsets. Let $$\mu^{[1]}=\int_{\Omega}\mu_{w}\times \mu_{w}dm(w)$$ where $$\mu=\int_{\Omega}\mu_{w}dm(w)$$ is the ergodic decomposition of $\mu$ with respect to $T$. Then we get a measure preserving system $(X^{[1]},\B^{\otimes [1]},\mu^{[1]},T^{[1]})$. Let $k\in \N$. If we have defined $\mu^{[k]}$, we can define $\mu^{[k+1]}$ as follows. Let $(\Omega_k,\mathcal{F}_k,m_k,S_k) $ denote the factor associated to $\mathcal{I}(T^{[k]})$ where $\mathcal{I}(T^{[k]})$ is a sub-$\sigma$-algebra of $\B^{\otimes [k]}$ consisting of all $T^{[k]}$-invariant subsets. Let $$\mu^{[k+1]}=\int_{\Omega_k}\mu^{[k]}_{w}\times \mu^{[k]}_{w}dm_{k}(w)$$ where $$\mu^{[k]}=\int_{\Omega_k}\mu^{[k]}_{w}dm_{k}(w)$$ is the ergodic decomposition of $\mu^{[k]}$ with respect to $T^{[k]}$. 
	
	For any $f\in L^{\infty}(\mu)$ and any $k\in \N$, we define the \emph{Host-Kra seminorm} by $$\nnorm f_{k}=\Big(\int_{X^{[k]}}\prod_{\underline{\epsilon} \in V_{k}}C^{|\underline{\epsilon}|}f(x_{\underline{\epsilon}})d\mu^{[k]}(\underline{x})\Big)^{\frac{1}{2^{k}}}$$ where $C$ is a complex conjugation operator and $|\underline{\epsilon}|=\epsilon_{1}+\cdots + \epsilon_{k}$.
	
	The follwoing results on Host-Kra seminorm will be used in the proof.
	\begin{thm}\label{thm1}
		$($\cite[Lemma 8.12, Proposition 8.16 and Theorem 9.7]{HK-book}$)$For a system $(X,\B,\mu,T)$, $k\ge 1$ and any $f\in L^{\infty}(\mu)$, we have :
		\begin{itemize}
			\item[(1)]$\nnorm f_{k+1}^{2^{k+1}}=\lim_{H\rightarrow \infty}\frac{1}{H}\sum_{h=0}^{H-1}\nnorm {f\cdot T^{h}\overline{f}}_{k}^{2^{k}}$;
			\item[(2)] $\nnorm f_{k}\le \nnorm f_{k+1}$;
			\item[(3)] there exists a factor $\mathcal{Z}_{k}(T)$, called $k$-step factor of $(X,\B,\mu,T)$, such that $\nnorm f_{k+1}=0$ if and only if $\mathbb{E}(f|\mathcal{Z}_{k}(T))=0$.
		\end{itemize} 
	\end{thm}
	\begin{rem}\label{rem}
		From the above, we can get a sequence of factors $\{\mathcal{Z}_{k}(T)\}_{k\ge 1}$ of $(X,\B,\mu,T)$. And, for any $k\in\N$, we have $\mathcal{Z}_{k}(T)\subset \mathcal{Z}_{k+1}(T)$. So we can define factor $\mathcal{Z}_{\infty}(T)$ called $\infty$-step factor by letting it be the smallest $\sigma$-algebra containing $\bigcup_{k\ge 1}\mathcal{Z}_{k}(T)$. $($For more details, see \cite[Chapter 8,9,16]{HK-book}$)$.
	\end{rem}
	\begin{prop}\label{prop2}
		$($\cite[Lemma 8]{L}$)$For a system $(X,\B,\mu,T)$, $k\ge 1$, $f\in L^{\infty}(\mu)$ which takes real value and $a\in \Z\backslash \{0\}$, one has $$\lim_{H\rightarrow \infty}\frac{1}{H}\sum_{h=0}^{H-1}\nnorm {f\cdot T^{ah}{f}}_{k}^{2^{k}}\le |a|\cdot\nnorm f_{k+1}^{2^{k+1}}.$$
	\end{prop}
	\begin{prop}\label{prop3}
		$($\cite[Proposition 5]{L}$)$For a system $(X,\B,\mu,T)$, $d\ge 1$, $f_{1},\cdots,f_{d}\in L^{\infty}(\mu)$ which take real value and are bounded by one and distinct $a_{1},\cdots,a_{d}\in \Z\backslash \{0\}$, there exists a constant $A(a_{1},\cdots,a_{d})>0$ such that $$\limsup_{N\rightarrow \infty}\norm{\frac{1}{N}\sum_{n=0}^{N-1}\prod_{i=1}^{d}T^{a_{i}n}f_{i}}_{2}\le A\min_{1\le i\le d}\nnorm {f_{i}}_{d+1}.$$
	\end{prop}
	Next, we reproduce local seminorm of bounded sequences which was defined by Host and Kra when they generalized Wiener–Wintner Theorem to nilsequences in \cite{HKU}.
	\begin{defn}
	$($\cite[Definition 2.1]{HKU}$)$Let $k\in \N$ and $\textbf{a}=\{a_n\}_{n\in \Z}$ be a bounded sequence. We say that \emph{$\textbf{a}$ satisfies $\mathcal{P}(k)$} if for any $\underline{h}=(h_1,\cdots,h_k)\in \Z^{k}$, the limit $$\lim_{N\rightarrow \infty}\frac{1}{N}\sum_{n=0}^{N-1}\prod_{\underline{\epsilon} \in V_{k}}C^{|\underline{\epsilon}|}a_{n+\underline{h}\cdot \underline{\epsilon}}$$ exists where $\underline{h}\cdot \underline{\epsilon}=\sum_{i=1}^{k}\epsilon_{i}h_{i}$. We denote the limit by $c_{\underline{h}}(\textbf{a})$.
	\end{defn}
	By \cite[Proposition 2.2]{HKU}, we know the limit $$\lim_{H\rightarrow \infty}\frac{1}{H^{k}}\sum_{\underline{h}\in [H]^{k}}c_{\underline{h}}(\textbf{a})$$ exists and is non-negative. Then we can define \emph{local seminorm} $\nnorm{\textbf{a}}_{k}$ of $\textbf{a}$ by $$\nnorm{\textbf{a}}_{k}=\Big(\lim_{H\rightarrow \infty}\frac{1}{H^{k}}\sum_{\underline{h}\in [H]^{k}}c_{\underline{h}}(\textbf{a})\Big)^{\frac{1}{2^k}}.$$
	
	The follwoing results on local seminorm will be used in the proof.
	\begin{prop}\label{prop4}
		$($\cite[Corollary 2.14]{HKU}$)$Let $k\ge 2$, $\textbf{b}=\{b_n\}_{n\in\Z}$ be a $(k-1)$-step nilsequence and $\delta>0$. There exists $c(\textbf{b},\delta)>0$ such that for any bounded sequence $\textbf{a}=\{a_n\}_{n\in \Z}$ which satisfies $\mathcal{P}(k)$, then $$\limsup_{N\rightarrow \infty}\Big|\frac{1}{N}\sum_{n=0}^{N-1}a_{n}b_{n}\Big|\le c\nnorm{\textbf{a}}_{k}+\delta \norm{\textbf{a}}_{\infty}$$ where $ \norm{\textbf{a}} _{\infty}=\sup_{n\in\Z}|a_n|$.
	\end{prop}
	\begin{prop}\label{prop5}
		$($\cite[Corollary 5.10]{HKU}$)$Let $k\ge 2$ and $\textbf{a}=\{a_n\}_{n\in \Z}$ be a bounded sequence. Assume that for any $\delta>0$, there exists a $(k-1)$-step nilsequence $\textbf{r}=\{r_n\}_{n\in\Z}$ such that $\textbf{a}-\textbf{r}$ satisfies $\mathcal{P}(k)$ and $\nnorm{\textbf{a}-\textbf{r}}_{k}<\delta$. Then for any $(k-1)$-step nilsequence $\textbf{b}=\{b_n\}_{n\in\Z}$, the limit $$\lim_{N\rightarrow \infty}\frac{1}{N}\sum_{n=0}^{N-1}a_{n}b_{n}$$ exists.
	\end{prop}
	\subsection{Furstenberg systems for sequences}
	Here, we recall the notion of the Furstenberg systems for bounded real sequences$($For details, see \cite{HN2018}$)$.
	\begin{defn}
		Let $\{N_{k}\}_{k\ge 1}$ be a strictly increasing sequence of positive integers and $I$ be a closed bounded interval on $\R$. We say that a sequence $\textbf{z}=\{z_n\}_{n\in \Z}$ which takes value on $I$ admits a correlation on sequence $\{N_{k}\}_{k\ge 1}$ if the limit $$\lim_{k\rightarrow \infty}\frac{1}{N_k}\sum_{n=1}^{N_k}\prod_{j=1}^{s}z_{n+n_j}$$ exists for any $s\in\N$ and any $(n_1,\cdots,n_s)\in \Z^{s}$.
	\end{defn}
	Let $\Omega=I^{\Z}$. The element of $\Omega$ can be written $w=(w(n))_{n\in\Z}$. So sequence $\textbf{z}$ can be viewed as an element of $\Omega$. The shift $\sigma$ on $\Omega$ is defined by $(\sigma(w))(n)=w(n+1)$ for any $n\in\Z$. If $\textbf{z}$ admits a correlation on sequence $\{N_{k}\}_{k\ge 1}$, then the limit $$\lim_{k\rightarrow \infty}\frac{1}{N_k}\sum_{n=1}^{N_k}\delta_{\sigma^{n}\textbf{z}}$$ exists in the weak-$*$ topology. We denote the limit by $\nu$.
	
		We say that system $(\Omega,\mathcal{F},\nu,\sigma)$ is the \emph{Furstenberg system} associated with $\textbf{z}$ on sequence $\{N_{k}\}_{k\ge 1}$ where $\mathcal{F}$ is the Borel $\sigma$-algebra of $\Omega$.
		
	Let $F_{0}$ the $0$-th coordinate projection of $\Omega$. That is, $F_{0}:\Omega\rightarrow I,w\mapsto w(0)$. Clearly, $$\lim_{k\rightarrow \infty}\frac{1}{N_k}\sum_{n=1}^{N_k}\prod_{j=1}^{s}z_{n+n_j}=\int \prod_{j=1}^{s}\sigma^{n_j}F_{0}d\nu$$ for any $s\in\N$ and any $(n_1,\cdots,n_s)\in \Z^{s}$.
	\subsection{Van der Corput's lemma} 
	\begin{lemma}\label{lem1}
		$($\cite[Lemma 1.3.1]{KN}$)$Let $\{u_n\}_{n\in \Z}$ be a bounded complex sequence. For any $N,H\ge 1$ with $H\le N$, we have $$H^{2}\Big|\sum_{n=1}^{N}u_{n}\Big|^{2}\le H(N+H-1)\sum_{n=1}^{N}|u_{n}|^{2}+2(N+H-1)\sum_{h=1}^{H-1}(H-h)\Re \sum_{n=1}^{N}u_{n}\overline{u}_{n+h}.$$
	\end{lemma}
	\section{Two lemmas}
	Before the statements of two lemmas, we introdue some notations. By $A\lesssim_{c_1,\cdots,c_d}B$, it means that there exists an implicit constant $C>0$ depending on $c_1,\cdots,c_d$ such that $A\le CB$. Let $X,Y$ be two non-empty sets and  $f:X\rightarrow \C,g:Y\rightarrow \C$ be two functions. By $f\otimes g$, it means a function $f\otimes g:X\times Y\rightarrow \C,(x,y)\mapsto f(x)g(y)$.
	\subsection{An estimation for cubic averages}
	Let $k\in \N$ and $V_{k}^{*}=V_{k}\backslash \{(0,\cdots,0)\}$. Define $\phi:V_{k}^{*}\rightarrow \{1,2,\cdots,2^{k}-1\},\underline{\epsilon}=(\epsilon_1,\cdots,\epsilon_k)\mapsto \sum_{i=1}^{k}\epsilon_{i}\cdot 2^{i-1}$. For any $1\le i\le k$, let $V_{k,i}^{*}=\{\underline{\epsilon}\in V_{k}^{*}:\epsilon_{i}=0\}$. For any $1\le i,j\le k$ with $i\neq j$, let $A^{j}_{i}=V_{k,i}^{*}\backslash V_{k,j}^{*}$.
	
	Let $\{a_{j,n}\}_{1\le j\le 2^{k}-1,n\ge 0}$ be a collection of bounded real sequences to be bounded by one. Let $N\in\N$. Define \emph{$k$-step cubic averages} $$C(N;a_{1,n},\cdots,a_{2^{k}-1,n})=\frac{1}{N^{k}}\sum_{\underline{h}\in [N]^{k}}\prod_{\underline{\epsilon}\in V_{k}^{*}}a_{\phi(\underline{\epsilon}),\underline{\epsilon}\cdot \underline{h}}.$$For example, when $k=3$,  $$C(N;a_{1,n},\cdots,a_{7,n})=\frac{1}{N^{3}}\sum_{\underline{h}\in [N]^{3}}a_{1,h_{1}}a_{2,h_{2}}a_{3,h_{1}+h_{2}}a_{4,h_{3}}a_{5,h_{3}+h_{1}}a_{6,h_{3}+h_{2}}a_{7,h_{3}+h_{2}+h_{1}}.$$
	
	When $k=2$, Assani built the following estimation.
	\begin{prop}\label{prop6}
		$($\cite[Lemma 5]{A}$)$Let $\{a_n\}_{n\in\Z},\{b_n\}_{n\in\Z},\{c_n\}_{n\in\Z}$ be three bound real sequences to be bounded by one. For any $N\in\N$, we have
		$$\Big|\frac{1}{N^{2}}\sum_{n,m=0}^{N-1}a_{n}b_{m}c_{n+m}\Big|^{2}\le \frac{1}{N}\sum_{n=0}^{N-1}\Big|\frac{1}{N}\sum_{m=0}^{N-1}b_{m}c_{n+m}\Big|^{2}\le \sup_{t}\Big|\frac{1}{N}\sum_{m=0}^{2(N-1)}e^{2\pi imt}c_{m}\Big|^{2}.$$
	\end{prop}
	Actually, Assani also gave an estimation for $3$-step cubic averages in \cite[Lemma 6]{A} and pointed out that one can build similar result for usual case. Here, we give the specific statements of general case and provide a proof for the completeness. The essence of the following is to let one index vanish.
	\begin{lemma}\label{lem2}
		Given $k,N\in \N$ with $k> 2$, let $\{a_{j,n}\}_{1\le j\le 2^{k}-1,n\in\Z}$ be a collection of bounded real sequences to be bounded by one. Then $$|C(N;a_{1,n},\cdots,a_{2^{k}-1,n})|^{2}\le \frac{C}{N^{k}}\sum_{h_{1},\cdots,h_{k-2}=0}^{N-1}\sup_{t}\Big|\sum_{h_{k}=0}^{N-1}e^{2\pi ih_{k}t}\prod_{\underline{\epsilon}\in A^{k}_{k-1} }a_{\phi(\underline{\epsilon}),\epsilon_{k}h_{k}+\sum_{i=1}^{k-2}\epsilon_{i}h_{i}}\Big|^{2}$$ where $C$ is an absolutely constant which can take value 2.
	\end{lemma}
	\begin{proof}
			\begin{align*}
		&|C(N;a_{1,n},\cdots,a_{2^{k}-1,n})|^{2}
		\\ &\le \frac{1}{N^{k-1}}\sum_{h_{1},\cdots,h_{k-1}=0}^{N-1}\Big|\frac{1}{N}\sum_{h_{k}=0}^{N-1}\prod_{\underline{\epsilon}\in V_{k}^{*}\backslash V_{k,k}^{*} }a_{\phi(\underline{\epsilon}),\sum_{i=1}^{k}\epsilon_{i}h_{i}}\Big|^{2}
		\\ &=
		\frac{1}{N^{k-1}}\sum_{h_{1},\cdots,h_{k-1}=0}^{N-1}\Big|\int_{\mathbb{T}}\Big(\frac{1}{N}\sum_{h_{k}=0}^{N-1}e^{-2\pi ih_{k}t}\prod_{\underline{\epsilon}\in A^{k}_{k-1} }a_{\phi(\underline{\epsilon}),\epsilon_{k}h_{k}+\sum_{i=1}^{k-2}\epsilon_{i}h_{i}}\Big)\cdot\\&\ \ \ \Big(\sum_{m=0}^{2(N-1)}e^{2\pi imt}\prod_{\underline{\epsilon}\in V_{k}^{*}\backslash (V_{k,k-1}^{*}\cup V_{k,k}^{*}) }a_{\phi(\underline{\epsilon}),m+\sum_{i=1}^{k-2}\epsilon_{i}h_{i}}\Big)e^{-2\pi ih_{k-1}t}dt\Big|^{2}\\&\le
		\frac{1}{N^{k-1}}\sum_{h_{1},\cdots,h_{k-2}=0}^{N-1}\int_{\mathbb{T}}\Big|\Big(\frac{1}{N}\sum_{h_{k}=0}^{N-1}e^{-2\pi ih_{k}t}\prod_{\underline{\epsilon}\in A^{k}_{k-1} }a_{\phi(\underline{\epsilon}),\epsilon_{k}h_{k}+\sum_{i=1}^{k-2}\epsilon_{i}h_{i}}\Big)\cdot\\&\ \ \ \Big(\sum_{m=0}^{2(N-1)}e^{2\pi imt}\prod_{\underline{\epsilon}\in V_{k}^{*}\backslash (V_{k,k-1}^{*}\cup V_{k,k}^{*}) }a_{\phi(\underline{\epsilon}),m+\sum_{i=1}^{k-2}\epsilon_{i}h_{i}}\Big)\Big|^{2}dt
		\\&\le
		\frac{1}{N^{k-1}}\sum_{h_{1},\cdots,h_{k-2}=0}^{N-1}\Big((2(N-1)+1)\sup_{t}\Big|\frac{1}{N}\sum_{h_{k}=0}^{N-1}e^{-2\pi ih_{k}t}\prod_{\underline{\epsilon}\in A^{k}_{k-1} }a_{\phi(\underline{\epsilon}),\epsilon_{k}h_{k}+\sum_{i=1}^{k-2}\epsilon_{i}h_{i}}\Big|^{2}\Big)\\&\le 
		\frac{2}{N^{k-2}}\sum_{h_{1},\cdots,h_{k-2}=0}^{N-1}\sup_{t}\Big|\frac{1}{N}\sum_{h_{k}=0}^{N-1}e^{-2\pi ih_{k}t}\prod_{\underline{\epsilon}\in A^{k}_{k-1} }a_{\phi(\underline{\epsilon}),\epsilon_{k}h_{k}+\sum_{i=1}^{k-2}\epsilon_{i}h_{i}}\Big|^{2}
		\end{align*} where the first inequality comes from Jesen's inequality and the second inequality comes from Parseval's identity. This finishes the proof.
	\end{proof}
	\subsection{An extimation for multivariable non-conventional ergodic averages}
		\begin{lemma}\label{lem3}
			Let $(X,\B,\mu,T)$ be a system, $d\in \N$, $f_{1},\cdots,f_{d}$ be bounded real valued functions to be bounded by one, $a_1,\cdots,a_d$ be distinct non-zero integers, and $\mu^{\mathcal{A}}_{d}$ be Furstenberg self-joining of $(X,\B,\mu,T)$ with respect to polynomial family $\mathcal{A}=\{a_{1}n,\cdots,a_{d}n\}$. Then $$\int_{X^{d}}\limsup_{N\rightarrow \infty}\sup_{t}\Big|\frac{1}{N}\sum_{n=1}^{N}e^{2\pi int}\prod_{i=1}^{d}f_{i}(T^{a_{i}n}x_{i})\Big|d\mu^{\mathcal{A}}_{d}(\textbf{x})\lesssim_{a_1,\cdots,a_d}\min_{1\le i\le d}\nnorm {f_{i}}_{d+3}.$$  
		\end{lemma}
		\begin{proof}
			Let $C>2$ and fix it. By Lemma \ref{lem1}, we have 
			\begin{align*}
			&\int_{X^{d}}\limsup_{N\rightarrow \infty}\sup_{t}\Big|\frac{1}{N}\sum_{n=1}^{N}e^{2\pi int}\prod_{i=1}^{d}f_{i}(T^{a_{i}n}x_{i})\Big|^{2}d\mu^{\mathcal{A}}_{d}(\textbf{x})
			\\&\le
			\frac{C}{H_{1}}+\frac{C}{H_{1}}\sum_{h_{1}=1}^{H_{1}-1}\int_{X^{d}} \limsup_{N\rightarrow \infty}\Big|\frac{1}{N}\sum_{n=1}^{N}\prod_{i=1}^{d}T^{a_{i}n}(f_{i}\cdot T^{a_{i}h_{1}}f_{i})(x_i)\Big|d\mu^{\mathcal{A}}_{d}(\textbf{x})
			\\&\le
			\frac{C}{H_{1}}+\frac{C}{H_{1}}\sum_{h_{1}=1}^{H_{1}-1}\Big(\int_{X^{d}} \limsup_{N\rightarrow \infty}\Big|\frac{1}{N}\sum_{n=1}^{N}\prod_{i=1}^{d}T^{a_{i}n}(f_{i}\cdot T^{a_{i}h_{1}}f_{i})(x_i)\Big|^{2}d\mu^{\mathcal{A}}_{d}(\textbf{x})\Big)^{\frac{1}{2}}
			\end{align*}
			for any $H_{1}\in \N$.
			By Proposition \ref{prop1}, Birkhoff's pointwise ergodic theorem and Lemma \ref{lem1}, we have 
			\begin{align*}
			&\int_{X^{d}}\limsup_{N\rightarrow \infty}\sup_{t}\Big|\frac{1}{N}\sum_{n=1}^{N}e^{2\pi int}\prod_{i=1}^{d}f_{i}(T^{a_{i}n}x_{i})\Big|^{2}d\mu^{\mathcal{A}}_{d}(\textbf{x})
			\\&\le
			\frac{C}{H_{1}}+\frac{C}{H_{1}}\sum_{h_{1}=1}^{H_{1}-1}\Big(\limsup_{N\rightarrow \infty}\int_{X^{d}} \Big|\frac{1}{N}\sum_{n=1}^{N}\prod_{i=1}^{d}T^{a_{i}n}(f_{i}\cdot T^{a_{i}h_{1}}f_{i})(x_i)\Big|^{2}d\mu^{\mathcal{A}}_{d}(\textbf{x})\Big)^{\frac{1}{2}}
			\\&\le
			\frac{C}{H_{1}}+\frac{C}{H_{1}}\sum_{h_{1}=1}^{H_{1}-1}\Big(\limsup_{N\rightarrow \infty}\int_{X^{d}}\frac{C}{H_{2}}+\frac{2(N+H_{2}-1)}{H_{2}N}\sum_{h_{2}=1}^{H_{2}-1}\frac{H_{2}-h_{2}}{H_{2}}\cdot \\&\ \ \ \frac{1}{N}\sum_{n=1}^{N}\prod_{i=1}^{d}T^{a_{i}n}(f_{i}\cdot T^{a_{i}h_{1}}f_{i})(x_i)T^{a_{i}(n+h_{2})}(f_{i}\cdot T^{a_{i}h_{1}}f_{i})(x_i)d\mu^{\mathcal{A}}_{d}(\textbf{x})\Big)^{\frac{1}{2}}
				\end{align*}
				for any $H_{2}\in \N$.
				By definition of $\mu^{\mathcal{A}}_{d}$, we know 
				\begin{align*}
				&\int_{X^{d}}\limsup_{N\rightarrow \infty}\sup_{t}\Big|\frac{1}{N}\sum_{n=1}^{N}e^{2\pi int}\prod_{i=1}^{d}f_{i}(T^{a_{i}n}x_{i})\Big|^{2}d\mu^{\mathcal{A}}_{d}(\textbf{x})
			\\&\le
			\frac{C}{H_{1}}+\frac{C}{H_{1}}\sum_{h_{1}=1}^{H_{1}-1}\Big(\frac{C}{H_{2}}+\limsup_{N\rightarrow \infty}\Big(\frac{2(N+H_{2}-1)}{H_{2}N}\sum_{h_{2}=1}^{H_{2}-1}\frac{H_{2}-h_{2}}{H_{2}}\cdot \\ &\ \ \ \frac{1}{N}\sum_{n=1}^{N}\lim_{M\rightarrow \infty} \frac{1}{M}\sum_{m=1}^{M}\int_{X}\prod_{i=1}^{d}T^{a_{i}(n+m)}(f_{i}\cdot T^{a_{i}h_{1}}f_{i}\cdot T^{a_{i}h_{2}}f_{i}\cdot T^{a_{i}(h_{1}+h_{2})}f_{i})d\mu(x)\Big)\Big)^{\frac{1}{2}}
			\\&\le
			\frac{C}{H_{1}}+\frac{C}{H_{1}}\sum_{h_{1}=1}^{H_{1}-1}\Big(\frac{C}{H_{2}}+\frac{C}{H_{2}}\sum_{h_{2}=1}^{H_{2}-1}\limsup_{N\rightarrow \infty} \frac{1}{N}\cdot\\&\ \ \ \sum_{n=1}^{N}\limsup_{M\rightarrow \infty} \norm{\frac{1}{M}\sum_{m=1}^{M}\prod_{i=1}^{d}T^{a_{i}(n+m)}(f_{i}\cdot T^{a_{i}h_{1}}f_{i}\cdot T^{a_{i}h_{2}}f_{i}\cdot T^{a_{i}(h_{1}+h_{2})}f_{i})}_{2}\Big)^{\frac{1}{2}}.
				\end{align*}
				By Proposition \ref{prop3}, we have 
				\begin{align*}
				&\int_{X^{d}}\limsup_{N\rightarrow \infty}\sup_{t}\Big|\frac{1}{N}\sum_{n=1}^{N}e^{2\pi int}\prod_{i=1}^{d}f_{i}(T^{a_{i}n}x_{i})\Big|^{2}d\mu^{\mathcal{A}}_{d}(\textbf{x})
			\\&\lesssim_{a_{1},\cdots,a_{d}}
			\frac{1}{H_{1}}+\frac{1}{H_{1}}\sum_{h_{1}=1}^{H_{1}-1}\Big(\frac{1}{H_{2}}+\frac{1}{H_{2}}\sum_{h_{2}=1}^{H_{2}-1}\min_{1\le i\le d}\nnorm {f_{i}\cdot T^{a_{i}h_{1}}f_{i}\cdot T^{a_{i}h_{2}}(f_{i}\cdot T^{a_{i}h_{1}}f_{i})}_{d+1}\Big)^{\frac{1}{2}}.
			\end{align*}
			By Proposition \ref{prop2} and taking $\limsup$ for $H_{1},H_{2}$ respectively, we have 
			\begin{align*}
			&\int_{X^{d}}\limsup_{N\rightarrow \infty}\sup_{t}\Big|\frac{1}{N}\sum_{n=1}^{N}e^{2\pi int}\prod_{i=1}^{d}f_{i}(T^{a_{i}n}x_{i})\Big|^{2}d\mu^{\mathcal{A}}_{d}(\textbf{x})
			\\&\lesssim_{a_{1},\cdots,a_{d} }\min_{1\le i\le d}
			\limsup_{H_{1}\rightarrow \infty}\frac{1}{H_{1}}\sum_{h_{1}=1}^{H_{1}-1}\Big(\limsup_{H_{2}\rightarrow \infty}\frac{1}{H_{2}}\sum_{h_{2}=1}^{H_{2}-1}\nnorm {(f_{i}T^{a_{i}h_{1}}f_{i})T^{a_{i}h_{2}}(f_{i} T^{a_{i}h_{1}}f_{i})}_{d+1}\Big)^{\frac{1}{2}}
			\\&\lesssim_{a_{1},\cdots,a_{d} }\min_{1\le i\le d}
			\limsup_{H_{1}\rightarrow \infty}\frac{1}{H_{1}}\sum_{h_{1}=1}^{H_{1}-1}\nnorm {f_{i}\cdot T^{a_{i}h_{1}}f_{i}}_{d+2}
			\\&\lesssim_{a_{1},\cdots,a_{d} }\min_{1\le i\le k}\nnorm {f_{i}}_{d+3}^{2}.
			\end{align*}
			
			Therefore, 
			$$\int_{X^{d}}\limsup_{N\rightarrow \infty}\sup_{t}\Big|\frac{1}{N}\sum_{n=1}^{N}e^{2\pi int}\prod_{i=1}^{d}f_{i}(T^{a_{i}n}x_{i})\Big|d\mu^{\mathcal{A}}_{d}(\textbf{x})\lesssim_{a_{1},\cdots,a_{d}}\min_{1\le i\le d}\nnorm {f_{i}}_{d+3}.$$ This finishes the proof.
		\end{proof}
	\section{Proof of Theorem \ref{A}}
	\begin{proof}[Proof of Theorem \ref{A}]
		Note that all constant sequences are nilsequences. So we only need to verify the following fact: Assume that for any $h_{1},\cdots,h_{d}\in L^{\infty}(\mu)$, we have $$\lim_{N\rightarrow \infty}\frac{1}{N}\sum_{n=0}^{N-1}\prod_{j=1}^{d}h_{j}(T^{a_{j}n}x)$$ exists almost everywhere. Then for any real valued bounded $f_{1},\cdots,f_{d}$ to be bounded by $\frac{1}{2}$, there exists full measure subset $X(f_{1},\cdots,f_{d})$ of $X$ such that for any $x\in X(f_{1},\cdots,f_{d})$, and any nilsequence $\textbf{b}=\{b_{n}\}_{n\in \Z}$, $$ \lim_{N\rightarrow \infty}\frac{1}{N}\sum_{n=0}^{N-1}b_{n}\prod_{j=1}^{d}f_{j}(T^{a_{j}n}x)$$ exists. 
		
		Choose a positive integer $k$ greater than one arbitrarily and fix it. Next, we show that the above fact holds for all $(k-1)$-step nilsequences. Once we have done it, we can finish the whole proof because the intersection of countable full measure subsets still is of full measure.
		
		For any $1\le j\le d$, let $$f_{j}=\mathbb{E}(f_{j}|\mathcal{Z}_{d+k}(T))+(f_{j}-\mathbb{E}(f_{j}|\mathcal{Z}_{d+k}(T))).$$ Based on this decomposition, the rest can be divided into two steps.
		\begin{step}
		At this step, we verify the following: If we can find $j$ from $\{1,\cdots,d\}$ such that $f_{j}$ has zero conditional expectation with respect to $\mathcal{Z}_{d+k}(T)$, then there exists full measure subset $X(f_{1},\cdots,f_{d})$ of $X$ such that for any $x\in X(f_{1},\cdots,f_{d})$, and any $(k-1)$-step nilsequence $\textbf{b}=\{b_{n}\}_{n\in \Z}$, $$ \lim_{N\rightarrow \infty}\frac{1}{N}\sum_{n=0}^{N-1}b_{n}\prod_{j=1}^{d}f_{j}(T^{a_{j}n}x)=0.$$
		\end{step}
		 The assumption tells us that there exists a full measure subset $X_1$ of $X$ such that for any $x\in X_1$, any $\underline{h}\in \Z^{k}$, we can define $c_{\underline{h}}(x)$ by $$c_{\underline{h}}(x)=\lim_{N\rightarrow \infty}\frac{1}{N}\sum_{n=0}^{N-1}\prod_{\underline{\epsilon}\in V_{k}}a_{n+\underline{h}\cdot \underline{\epsilon}}(x)$$ where for any $n\in\Z$, $a_{n}(x)=\prod_{j=1}^{d}f_{j}(T^{a_{j}n}x)$. By \cite[Proposition 2.2]{HKU}, for any $x\in X_{1}$, we can define $\nnorm {\textbf{a}(x)}_{k}$ by  
		$$\nnorm {\textbf{a}(x)}_{k}=\Big(\lim_{H\rightarrow \infty}\frac{1}{H^k}\sum_{\underline{h}\in [H]^{k}}c_{\underline{h}}(x)\Big)^{\frac{1}{2^{k}}}.$$
		Next, we verify that for $\mu$-a.e. $x\in X_{1}$, $\nnorm{\textbf{a}(x)}_{k}=0$. The rest proof of this step is divided into two parts.
		
		\textbf{Part I}: $k>2$.  Let $\mu^{\mathcal{A}}_{d}$ be Furstenberg self-joining of $(X,\B,\mu,T)$ with respect to polynomial family $\mathcal{A}=\{a_{1}n,\cdots,a_{d}n\}$. By the definition of $\mu^{\mathcal{A}}_{d}$ and \cite[Proposition 2.2]{HKU}, we have 
		\begin{align*}
		&\int_{X}\lim_{H\rightarrow \infty}\frac{1}{H^k}\sum_{\underline{h}\in [H]^{k}}c_{\underline{h}}(x)d\mu(x)
		\\&=
		\lim_{H\rightarrow \infty}\frac{1}{H^k}\sum_{\underline{h}\in [H]^{k}}\int_{X}\lim_{N\rightarrow \infty}\frac{1}{N}\sum_{n=0}^{N-1}\prod_{\underline{\epsilon}\in V_{k}}a_{n+\underline{h}\cdot \underline{\epsilon}}(x)d\mu(x)
		\\&=
		\lim_{H\rightarrow \infty}\frac{1}{H^k}\sum_{\underline{h}\in [H]^{k}}\lim_{N\rightarrow \infty}\frac{1}{N}\sum_{n=0}^{N-1}\int_{X}\prod_{j=1}^{d}T^{a_{j}n}(\prod_{\underline{\epsilon}\in V_{k}}T^{a_{j}(\underline{h}\cdot \underline{\epsilon})}f_{j})d\mu(x)
		\\&=
		\lim_{H\rightarrow \infty}\frac{1}{H^k}\sum_{\underline{h}\in [H]^{k}}\int_{X^{d}}(\prod_{\underline{\epsilon}\in V_{k}}T^{a_{1}(\underline{h}\cdot \underline{\epsilon})}f_{1})\otimes \cdots \otimes (\prod_{\underline{\epsilon}\in V_{k}}T^{a_{d}(\underline{h}\cdot \underline{\epsilon})}f_{d})d\mu^{\mathcal{A}}_{d}(\textbf{x}).
		\end{align*}
		By Fatou's Lemma, we know 
		\begin{align*}
		&\int_{X}\lim_{H\rightarrow \infty}\frac{1}{H^k}\sum_{\underline{h}\in [H]^{k}}c_{\underline{h}}(x)d\mu(x)
		\\&\le 
		\Big(\int_{X^{d}}\limsup_{H\rightarrow \infty}\Big|\frac{1}{H^k}\sum_{\underline{h}\in [H]^{k}}(\prod_{\underline{\epsilon}\in V_{k}}T^{a_{1}(\underline{h}\cdot \underline{\epsilon})}f_{1})\otimes \cdots \otimes (\prod_{\underline{\epsilon}\in V_{k}}T^{a_{d}(\underline{h}\cdot \underline{\epsilon})}f_{d})\Big|^{2^{k-2}}d\mu^{\mathcal{A}}_{d}(\textbf{x})\Big)^{\frac{1}{2^{k-2}}}.
		\end{align*}
		Note that 
		\begin{align*}
		&\frac{1}{H^k}\sum_{\underline{h}\in [H]^{k}}(\prod_{\underline{\epsilon}\in V_{k}}T^{a_{1}(\underline{h}\cdot \underline{\epsilon})}f_{1})\otimes \cdots \otimes (\prod_{\underline{\epsilon}\in V_{k}}T^{a_{d}(\underline{h}\cdot \underline{\epsilon})}f_{d})
		\\&=
		\Big(f_{1}\otimes \cdots \otimes f_{d}\Big)\frac{1}{H^{k}}\sum_{\underline{h}\in [H]^{k}}\prod_{\underline{\epsilon}\in V_{k}^{*}}G_{\phi(\underline{\epsilon}),\underline{h}\cdot \underline{\epsilon}}
		\end{align*}
		where $\{G_{i,n}\}_{1\le i\le 2^{k}-1,n\in\Z}$ is a family of function sequences and for any $1\le i\le 2^{k}-1,n\in\Z$, $G_{i,n}=T^{a_{1}n}f_{1}\otimes \cdots\otimes T^{a_{d}n}f_{d}$.
		
		Let $C>2,\tilde{C}>(1024)^{k}C^{1024k}$ and fix them. 
		By Lemma \ref{lem2}, we know that 
		\begin{align*}
		&\Big(\frac{1}{H^k}\sum_{\underline{h}\in [H]^{k}}(\prod_{\underline{\epsilon}\in V_{k}}T^{a_{1}(\underline{h}\cdot \underline{\epsilon})}f_{1})\otimes \cdots \otimes (\prod_{\underline{\epsilon}\in V_{k}}T^{a_{d}(\underline{h}\cdot \underline{\epsilon})}f_{d})\Big)^{2}
		\\&\le
		\frac{C}{H^{k-2}}\sum_{h_{1},\cdots,h_{k-2}=0}^{H-1}\sup_{t}\Big|\frac{1}{H}\sum_{h_{k}=0}^{H-1}e^{2\pi ih_{k}t}\prod_{\underline{\epsilon}\in A_{k-1}^{k}}G_{\phi(\underline{\epsilon}),h_{k}\epsilon_{k}+\sum_{i=1}^{k-2}h_{i}\epsilon_{i}}\Big|^{2}.
		\end{align*}
		By an easy computation, for $k\ge 4$, we have 
		\begin{align*}
		& H\Big|\frac{1}{H^k}\sum_{\underline{h}\in [H]^{k}}(\prod_{\underline{\epsilon}\in V_{k}}T^{a_{1}(\underline{h}\cdot \underline{\epsilon})}f_{1})\otimes \cdots \otimes (\prod_{\underline{\epsilon}\in V_{k}}T^{a_{d}(\underline{h}\cdot \underline{\epsilon})}f_{d})\Big|^{2^{k-2}}\\&\le \sum_{h_{1}=0}^{H-1}
		\Big(\frac{1}{H}\sum_{h_{2}=0}^{H-1}\Big(\cdots \Big(\frac{1}{H}\sum_{h_{k-2}=0}^{H-1}\sup_{t}\Big|\frac{1}{H}\sum_{h_{k}=0}^{H-1}e^{2\pi ih_{k}t}\cdot\\&\ \ \ \prod_{\underline{\epsilon}\in A_{k-1}^{k}}G_{\phi(\underline{\epsilon}),h_{k}\epsilon_{k}+\sum_{i=1}^{k-2}h_{i}\epsilon_{i}}\Big|^{2}\Big)^{2}\cdots \Big)^{2}\Big)^{2}.
		\end{align*}
		Choose sufficiently large $H$ and fix it. By Lemma \ref{lem1}, we have for any $1\le K\le H$ and any $(h_{1},\cdots,h_{k-2})\in [H]^{k-2}$,
		\begin{align*}
		&\sup_{t}\Big|\frac{1}{H}\sum_{h_{k}=0}^{H-1}e^{2\pi ih_{k}t}\prod_{\underline{\epsilon}\in A_{k-1}^{k}}G_{\phi(\underline{\epsilon}),h_{k}\epsilon_{k}+\sum_{i=1}^{k-2}h_{i}\epsilon_{i}}\Big|^{2}
		\\&\le
		\frac{C}{K}+\frac{C}{K}\sum_{s=1}^{K}\Big|\frac{1}{H}\sum_{n=0}^{H-1}\Big(\prod_{\underline{\epsilon}\in A_{k-1}^{k}}G_{\phi(\underline{\epsilon}),n+\sum_{i=1}^{k-2}h_{i}\epsilon_{i}}\Big)\cdot \Big(\prod_{\underline{\epsilon}\in A_{k-1}^{k}}G_{\phi(\underline{\epsilon}),n+s+\sum_{i=1}^{k-2}h_{i}\epsilon_{i}}\Big)\Big|.
		\end{align*}
		Then we have
		\begin{align*}
		&\frac{1}{H}\sum_{h_{k-2}=0}^{H-1}\sup_{t}\Big|\frac{1}{H}\sum_{h_{k}=0}^{H-1}e^{2\pi ih_{k}t}\prod_{\underline{\epsilon}\in A_{k-1}^{k}}G_{\phi(\underline{\epsilon}),h_{k}\epsilon_{k}+\sum_{i=1}^{k-2}h_{i}\epsilon_{i}}\Big|^{2}
		\\&\le
		\frac{C}{K}+\frac{C}{K}\sum_{s=1}^{K}\frac{1}{H}\sum_{h_{k-2}=0}^{H-1}\Big|\frac{1}{H}\sum_{n=0}^{H-1}\prod_{\underline{\epsilon}\in A_{k-1}^{k}}G_{\phi(\underline{\epsilon}),n+\sum_{i=1}^{k-2}h_{i}\epsilon_{i}}\prod_{\underline{\epsilon}\in A_{k-1}^{k}}G_{\phi(\underline{\epsilon}),n+s+\sum_{i=1}^{k-2}h_{i}\epsilon_{i}}\Big|
		\\&\le
		\frac{C}{K}+\frac{C}{K}\sum_{s=1}^{K}\Big(\frac{1}{H}\sum_{h_{k-2}=0}^{H-1}\Big|\frac{1}{H}\sum_{n=0}^{H-1}\prod_{\underline{\epsilon}\in A_{k-1}^{k}}G_{\phi(\underline{\epsilon}),n+\sum_{i=1}^{k-2}h_{i}\epsilon_{i}} \prod_{\underline{\epsilon}\in A_{k-1}^{k}}G_{\phi(\underline{\epsilon}),n+s+\sum_{i=1}^{k-2}h_{i}\epsilon_{i}}\Big|^{2}\Big)^{\frac{1}{2}}
		\end{align*}
		for any $(h_{1},\cdots,h_{k-3})\in [H]^{k-3}$.
		For any $1\le s\le K$, by Proposition \ref{prop6}, we know 
		\begin{align*}
		&\frac{1}{H}\sum_{h_{k-2}=0}^{H-1}\Big|\frac{1}{H}\sum_{n=0}^{H-1}\prod_{\underline{\epsilon}\in A_{k-1}^{k}}G_{\phi(\underline{\epsilon}),n+\sum_{i=1}^{k-2}h_{i}\epsilon_{i}}\prod_{\underline{\epsilon}\in A_{k-1}^{k}}G_{\phi(\underline{\epsilon}),n+s+\sum_{i=1}^{k-2}h_{i}\epsilon_{i}}\Big|^{2}
		\\&\le
		\sup_{t}\Big|\frac{1}{H}\sum_{m=0}^{2(H-1)}e^{2\pi imt}\Big(\prod_{\underline{\epsilon}\in V_{k-3}^{*}}T^{a_{1}(m+\sum_{i=1}^{k-3}h_{i}\epsilon_{i})}f_{1}\otimes \cdots \otimes \prod_{\underline{\epsilon}\in V_{k-3}^{*}}T^{a_{d}(m+\sum_{i=1}^{k-3}h_{i}\epsilon_{i})}f_{d}\Big)\cdot \\&\ \ \  \Big(\prod_{\underline{\epsilon}\in V_{k-3}^{*}}T^{a_{1}(m+s+\sum_{i=1}^{k-3}h_{i}\epsilon_{i})}f_{1}\otimes \cdots \otimes \prod_{\underline{\epsilon}\in V_{k-3}^{*}}T^{a_{d}(m+s+\sum_{i=1}^{k-3}h_{i}\epsilon_{i})}f_{d}\Big)\Big|^{2}.
		\end{align*}
		
		Therefore,
		\begin{align*}
		&\frac{1}{H}\sum_{h_{k-2}=0}^{H-1}\sup_{t}\Big|\frac{1}{H}\sum_{h_{k}=0}^{H-1}e^{2\pi ih_{k}t}\prod_{\underline{\epsilon}\in A_{k-1}^{k}}G_{\phi(\underline{\epsilon}),h_{k}\epsilon_{k}+\sum_{i=1}^{k-2}h_{i}\epsilon_{i}}\Big|^{2}
		\\&\le
		\frac{C}{K}+\frac{C}{K}\sum_{s=1}^{K}	\sup_{t}\Big|\frac{1}{H}\sum_{m=0}^{2(H-1)}e^{2\pi imt}\cdot\\&\ \ \ \Big(\prod_{\underline{\epsilon}\in V_{k-3}^{*}}T^{a_{1}(m+\sum_{i=1}^{k-3}h_{i}\epsilon_{i})}f_{1}\otimes \cdots \otimes \prod_{\underline{\epsilon}\in V_{k-3}^{*}} T^{a_{d}(m+\sum_{i=1}^{k-3}h_{i}\epsilon_{i})}f_{d}\Big) \cdot\\&\ \ \  \Big(\prod_{\underline{\epsilon}\in V_{k-3}^{*}}T^{a_{1}(m+s+\sum_{i=1}^{k-3}h_{i}\epsilon_{i})}f_{1}\otimes \cdots \otimes \prod_{\underline{\epsilon}\in V_{k-3}^{*}}T^{a_{d}(m+s+\sum_{i=1}^{k-3}h_{i}\epsilon_{i})}f_{d}\Big)\Big|.
		\end{align*}
		If $k=3$, by Proposition \ref{prop2}, Lemma \ref{lem3} and taking $\limsup$ for $H,K$ in turn, we have
		$$
		\int_{X}\lim_{H\rightarrow \infty}\frac{1}{H^3}\sum_{\underline{h}\in [H]^{3}}c_{\underline{h}}(x)d\mu(x)\lesssim_{a_{1},\cdots,a_{d} }\min_{1\le i\le d}\nnorm {f_{i}}_{d+4}.
		$$ If $k>3$, we go on working. 
		
		For any $(h_{1},\cdots,h_{k-4})\in [H]^{k-4}$, by Lemma \ref{lem1} and Proposition \ref{prop6}, we have
		\begin{align*}
		&\frac{1}{H}\sum_{h_{k-3}=0}^{H-1}\Big(\frac{1}{H}\sum_{h_{k-2}=0}^{H-1}\sup_{t}\Big|\frac{1}{H}\sum_{h_{k}=0}^{H-1}e^{2\pi ih_{k}t}\prod_{\underline{\epsilon}\in A_{k-1}^{k}}G_{\phi(\underline{\epsilon}),h_{k}\epsilon_{k}+\sum_{i=1}^{k-2}h_{i}\epsilon_{i}}\Big|^{2}\Big)^{2}
		\\&\le
		\frac{C^{2}}{K^{2}}+\frac{4C^{2}}{K}+\frac{4C^{3}}{K_{1}}+\frac{4C^{2}}{K}\sum_{s=1}^{K}\frac{C}{K_{1}}\sum_{t_{1}=1}^{K_{1}}\sup_{t}\Big|\frac{1}{2H-1}\sum_{m=0}^{4(H-1)}e^{2\pi imt}\cdot\\&\ \ \ 
		\Big(\prod_{\underline{\epsilon}\in V_{k-4}^{*}}T^{a_{1}(m+\sum_{i=1}^{k-4}h_{i}\epsilon_{i})}f_{1}\otimes \cdots \otimes \prod_{\underline{\epsilon}\in V_{k-4}^{*}} T^{a_{d}(m+\sum_{i=1}^{k-4}h_{i}\epsilon_{i})}f_{d}\Big)\cdot \\&\ \ \  \Big(\prod_{\underline{\epsilon}\in V_{k-4}^{*}}T^{a_{1}(m+t_{1}+\sum_{i=1}^{k-4}h_{i}\epsilon_{i})} f_{1}\otimes \cdots \otimes \prod_{\underline{\epsilon}\in V_{k-4}^{*}} T^{a_{d}(m+t_{1}+\sum_{i=1}^{k-4}h_{i}\epsilon_{i})}f_{d}\Big)\cdot \\&\ \ \  \Big(\prod_{\underline{\epsilon}\in V_{k-4}^{*}}T^{a_{1}(m+s+\sum_{i=1}^{k-4}h_{i}\epsilon_{i})}f_{1}\otimes \cdots \otimes \prod_{\underline{\epsilon}\in V_{k-4}^{*}}T^{a_{d}(m+s+\sum_{i=1}^{k-4}h_{i}\epsilon_{i})}f_{d}\Big)\cdot \\&\ \ \ \Big(\prod_{\underline{\epsilon}\in V_{k-4}^{*}}T^{a_{1}(m+s+t_{1}+\sum_{i=1}^{k-4}h_{i}\epsilon_{i})}f_{1}\otimes \cdots \otimes \prod_{\underline{\epsilon}\in V_{k-4}^{*}}T^{a_{d}(m+s+t_{1}+\sum_{i=1}^{k-4}h_{i}\epsilon_{i})}f_{d}\Big)
		\Big|
		\end{align*}
		for any $1\le K_{1}\le H$.
		In the above process, the index $h_{k-3}$ vanishes and new index $t_{1}$ is introduced.
	If $k=4$, we can stop at here. If $k>4$, we repeat the process until index $h_{1}$ vanishes.
		Finally, if $k\ge 4$, for any $1\le K,K_{1},\cdots,K_{k-3}\le H$, we have 
		\begin{align*}
		&\Big|\frac{1}{H^k}\sum_{\underline{h}\in [H]^{k}}(\prod_{\underline{\epsilon}\in V_{k}}T^{a_{1}(\underline{h}\cdot \underline{\epsilon})}f_{1})\otimes \cdots \otimes (\prod_{\underline{\epsilon}\in V_{k}}T^{a_{d}(\underline{h}\cdot \underline{\epsilon})}f_{d})\Big|^{2^{k-2}}
		\\&\le
		R(K,K_{1},\cdots,K_{k-3})+\tilde{C}\frac{1}{K}\sum_{s=1}^{K}\frac{1}{K_{1}}\sum_{t_{1}=1}^{K_{1}}\cdots \frac{1}{K_{k-3}}\sum_{t_{k-3}=1}^{K_{k-3}}
		\sup_{t}\Big|\frac{1}{2^{k-2}(H-1)+1}\cdot\\&\ \ \ \sum_{m=0}^{2^{k-2}(H-1)}e^{2\pi imt}\Big(\prod_{\underline{\epsilon}\in V_{k-2}}T^{a_{1}(m+\underline{\epsilon}\cdot \underline{k})}f_{1}\Big)\otimes \cdots \otimes \Big(\prod_{\underline{\epsilon}\in V_{k-2}}T^{a_{d}(m+\underline{\epsilon}\cdot \underline{k})}f_{d}\Big)\Big|
		\end{align*}
 where $\underline{k}=(t_{k-3},\cdots,t_{1},s),$ and $$\lim_{K\rightarrow \infty}\cdots \lim_{K_{k-3}\rightarrow \infty}R(K,K_{1},\cdots,K_{k-3})=0.$$ 
		
		By Proposition \ref{prop2}, Lemma \ref{lem3} and taking $\limsup$ for $H,K_{k-3},\cdots,K_{1},K$ in turn, we have
		$$
		\int_{X}\lim_{H\rightarrow \infty}\frac{1}{H^k}\sum_{\underline{h}\in [H]^{k}}c_{\underline{h}}(x)d\mu(x)\lesssim_{a_{1},\cdots,a_{d} }\min_{1\le i\le d}\nnorm {f_{i}}_{d+k+1}.
		$$
		\textbf{Part II}: $k=2$. By Proposition \ref{prop6} and Lemma \ref{lem3}, we have
		\begin{align*}
		&\int_{X}\lim_{H\rightarrow \infty}\frac{1}{H^k}\sum_{\underline{h}\in [H]^{k}}c_{\underline{h}}(x)d\mu(x)
		\\&\le
		\Big(\int_{X^{d}}\limsup_{H\rightarrow \infty}\Big|\frac{1}{H^2}\sum_{\underline{h}\in [H]^{2}}(\prod_{\underline{\epsilon}\in V_{2}}T^{a_{1}(\underline{h}\cdot \underline{\epsilon})}f_{1})\otimes \cdots \otimes (\prod_{\underline{\epsilon}\in V_{2}}T^{a_{d}(\underline{h}\cdot \underline{\epsilon})}f_{d})\Big|^{2}d\mu^{\mathcal{A}}_{d}(\textbf{x})\Big)^{\frac{1}{2}}
		\\&\le
		\Big(\int_{X^{d}}\limsup_{H\rightarrow \infty}\Big|\frac{1}{H^2}\sum_{h_{1},h_{2}=0}^{H-1}(T^{a_{1}h_{1}}f_{1}T^{a_{1}h_{2}}f_{1}T^{a_{1}(h_{1}+h_{2})}f_{1})\otimes \cdots \otimes \\&\ \ \  (T^{a_{d}h_{1}}f_{d}T^{a_{d}h_{2}}f_{d}T^{a_{d}(h_{1}+h_{2})}f_{d})\Big|^{2}d\mu^{\mathcal{A}}_{d}(\textbf{x})\Big)^{\frac{1}{2}}
		\\&\le
		\Big(4\int_{X^{d}}\limsup_{H\rightarrow \infty}\sup_{t}\Big|\frac{1}{2H-1}\sum_{m=0}^{2(H-1)}e^{2\pi imt}(T^{a_{1}m}f_{1})\otimes \cdots \otimes  (T^{a_{d}m}f_{d})\Big|^{2}d\mu^{\mathcal{A}}_{d}(\textbf{x})\Big)^{\frac{1}{2}}
		\\&\lesssim_{a_{1},\cdots,a_{d} }\min_{1\le i\le d}\nnorm {f_{i}}_{d+3}.
		\end{align*}
		To sum up, for any $k\ge 2$, we have 
		$$
		\int_{X}\lim_{H\rightarrow \infty}\frac{1}{H^k}\sum_{\underline{h}\in [H]^{k}}c_{\underline{h}}(x)d\mu(x)\lesssim_{a_{1},\cdots,a_{d} }\min_{1\le i\le d}\nnorm {f_{i}}_{d+k+1}.
		$$
		
		By Theorem \ref{thm1}, we know that for $\mu$-a.e. $x\in X_{1}$, $\nnorm{\textbf{a}(x)}_{k}=0$. By Proposition \ref{prop4}, we know that there exists full measure subset $X(f_{1},\cdots,f_{d})$ of $X$ such that for any $x\in X(f_{1},\cdots,f_{d})$, and any $(k-1)$-step nilsequence $\textbf{b}=\{b_{n}\}_{n\in \Z}$, $$ \lim_{N\rightarrow \infty}\frac{1}{N}\sum_{n=0}^{N-1}b_{n}\prod_{j=1}^{d}f_{j}(T^{a_{j}n}x)=0.$$
		
		\begin{step}
			At this step, we verify the following: If all $f_{j},1\le j\le d$ are measurable with respect to $\mathcal{Z}_{d+k}(T)$, then there exists full measure subset $X(f_{1},\cdots,f_{d})$ of $X$ such that for any $x\in X(f_{1},\cdots,f_{d})$, and any $(k-1)$-step nilsequence $\textbf{b}=\{b_{n}\}_{n\in \Z}$, $$ \lim_{N\rightarrow \infty}\frac{1}{N}\sum_{n=0}^{N-1}b_{n}\prod_{j=1}^{d}f_{j}(T^{a_{j}n}x)$$ exists.
		\end{step}
			By \cite[Theorem 16.10]{HK-book}, we know that for any $1\le j\le d$ and any $r\in \N$, we can find a function sequence $\{g_{j,m}\}_{m\ge 1}$ such that for each $m\ge 1$, the followings hold:
			\begin{itemize}
				\item[(1)]$\norm {g_{j,m}}_{\infty}\le \norm {f_{j}}_{\infty}$;
				\item[(2)]$\norm {g_{j,m}-f_{j}}_{1}\le \frac{1}{4^{m+1+r}Cd}$ where $C=d^{2^{k+d+1}}$;
				\item[(3)]for $\mu$-a.e. $x\in X$, $\{g_{j,m}(T^{a_{j}n}x)\}_{n\in \Z}$ is a $(d+k)$-step nilsequence.
			\end{itemize}
			Let $s=d+k+1$. By Birkhoff's pointwise ergodic theorem and our assumption, we can find a full measure subset $X_{0}$ of $X$ satisfies the followings:
			\begin{itemize}
				\item[(1)]for any $x\in X_{0}$, any $n\in \Z$ and any $1\le j\le d$, $|f_{j}(T^{a_{j}n}x)|\le \frac{1}{2}$;
				\item[(2)]for any $x\in X_{0}$, any $n\in \Z$, any $1\le j\le d$ and any $m\ge 1$, $|g_{j,m}(T^{a_{j}n}x)|\le  \frac{1}{2}$;
				\item[(3)]for any $x\in X_{0}$, any $1\le j\le d$ and any $m\ge 1$, $\{g_{j,m}(T^{a_{j}n}x)\}_{n\in \Z}$ is a $(d+k)$-step nilsequence;
				\item[(4)]for any $x\in X_{0}$, any $1\le j\le d$ and any $m\ge 1$, we have $$\lim_{N\rightarrow \infty}\frac{1}{N}\sum_{n=0}^{N-1}|f_{j}(T^{a_{j}n}x)-g_{j,m}(T^{a_{j}n}x)|=\mathbb{E}(|f_{j}-g_{j,m}|\big|\mathcal{I}_{a_{j}})(x)$$ where $\mathcal{I}_{a_{j}}$ is a sub-$\sigma$-algebra generated by all $T^{a_{j}}$-invariant subsets;
				\item[(5)]for any $x\in X_{0}$, any $m\ge 1$ and any $\underline{h}\in \Z^{s}$, the limit $$\lim_{N\rightarrow \infty}\frac{1}{N}\sum_{n=0}^{N-1}\prod_{\underline{\epsilon} \in V_{s}}a^{(m)}_{n+\underline{h}\cdot \underline{\epsilon}}$$ exists where 
				\begin{align*}
				a_{n}^{(m)}&=
				\prod_{j=1}^{d}f_{j}(T^{a_{j}n}x)-\prod_{j=1}^{d}g_{j,m}(T^{a_{j}n}x)
				\\&=
				\sum_{i=1}^{d}\prod_{j=1}^{i-1}f_{j}(T^{a_{j}n}x)\cdot (f_{i}(T^{a_{i}n}x)-g_{i,m}(T^{a_{i}n}x))\cdot \prod_{l=i+1}^{d}g_{l,m}(T^{a_{l}n}x)
				\end{align*}
				 for any $n\in\Z$. we denote the limit by $c_{\underline{h}}^{(m)}(x)$;
				\item[(6)]for any $x\in X_{0}$, any $m\ge 1$ and any $\underline{h}\in \Z^{s}$, $$|c_{\underline{h}}^{(m)}(x)|\le C\sum_{j=1}^{d}\mathbb{E}\Big(|f_{j}-g_{j,m}|\big|\mathcal{I}_{a_{j}}\Big)(x).$$
			\end{itemize}
			By \cite[Proposition 2.2]{HKU}, for any $m\ge 1$, we have 
			\begin{align*}
			&\int_{X}\lim_{H\rightarrow \infty}\frac{1}{H^s}\sum_{\underline{h}\in [H]^{s}}c^{(m)}_{\underline{h}}(x)d\mu(x)
			\\&=
			\lim_{H\rightarrow \infty}\frac{1}{H^s}\sum_{\underline{h}\in [H]^{s}}\int_{X}c^{(m)}_{\underline{h}}(x)d\mu(x)
			\\&\le
			\limsup_{H\rightarrow \infty}\frac{1}{H^s}\sum_{\underline{h}\in [H]^{s}}\int_{X}C\sum_{j=1}^{d}\mathbb{E}\big(|f_{j}-g_{j,m}|\big|\mathcal{I}_{a_{j}}\Big)(x)d\mu(x)
			\\&\le
			C\sum_{j=1}^{d}\norm {f_{j}-g_{j,m}}_{1}
			\\&=
			\frac{1}{4^{m+r+1}}. 
			\end{align*}
			Then for any $m\ge 1$, $$\mu(\{x\in X_{0}: \lim_{H\rightarrow \infty}\frac{1}{H^s}\sum_{\underline{h}\in [H]^{s}}c^{(m)}_{\underline{h}}(x)< \frac{1}{2^{m+1+r}}\})\ge 1-\frac{1}{2^{m+1+r}}.$$ Let $$X_{m}=\{x\in X_{0}: \lim_{H\rightarrow \infty}\frac{1}{H^s}\sum_{\underline{h}\in [H]^{s}}c^{(m)}_{\underline{h}}(x)\le \frac{1}{2^{m+1+r}}\}.$$
			
			For any $m\ge 1$, let $Y_{r}=\bigcap_{m\ge 1}X_{m}$, then $\mu(Y_{r})\ge 1-\frac{1}{2^{r}}$. Let $Y=\bigcup_{r\ge 1}Y_{r}$. Then $\mu(Y)=1$. 
			
			Note that the product of finite $(d+k)$-step nilsequences is still a $(d+k)$-step nilsequence and every $(k-1)$-step nilsequence can be viewed as a $(d+k)$-step nilsequence.
			
			By the definition of $Y$ and Proposition \ref{prop5}, we know that for any $x\in Y$, and any $(k-1)$-step nilsequence $\textbf{b}=\{b_{n}\}_{n\in \Z}$, $$ \lim_{N\rightarrow \infty}\frac{1}{N}\sum_{n=0}^{N-1}b_{n}\prod_{j=1}^{d}f_{j}(T^{a_{j}n}x)$$ exists. This finishes the whole proof.
	\end{proof}
	\section{Proof of Theorem \ref{B}}
	Before the proof, we provide a lemma. Based on \cite[Theorem 2.8]{FN2022} and \cite[Theorem 1.1]{W}, we can get the following lemma by repeating the arguments of the proof of \cite[Proposition 3.2]{HN}.
	\begin{lemma}\label{lem4}
	Given $d\in \N$, let $S_{1},\cdots,S_{d}$ be invertible measure preserving transformations acting on a Lebesgue space $(X,\B,\mu)$ such that $S_{1},\cdots,S_{d}$ are commuting and let $p_{1}(n),\cdots,p_{d}(n)$ be non-constant integer coefficents polynomials with $\deg p_{1}>\deg p_{2}>\cdots >\deg p_{d}\ge 2$. Then for any real valued $g_{1},\cdots,g_{d}\in L^{\infty}(\mu)$, and any strictly increasing sequence of positive integers $\{N_{k}\}_{k\ge 1}$, there exists a subsequence $\{N'_{k}\}_{k\ge 1}$ of $\{N_{k}\}_{k\ge 1}$ such that for $\mu$-a.e $x\in X$, the sequence $\{\prod_{j=1}^{d}g_{j}(S_{j}^{p_{j}(n)}x)\}_{n\in \Z}$ admits a correlation on sequence $\{N'_{k}\}_{k\ge 1}$. Moreover, if for some $j\in \{1,\cdots,d\}$, $\mathbb{E}(g_{j}|\mathcal{Z}_{\infty}(S_{j}))=0$, then the $0$-th coordinate projection $F_{0}:\Omega\rightarrow \R$ has zero conditional expectation with respect to the Pinsker factor\footnote{Every system $(Y,\mathcal{D},\nu,S)$ has a maximal zero entropy factor. The factor is called Pinsker factor of $(Y,\mathcal{D},\nu,S)$.} of the corresponding Furstenberg system.
	\end{lemma}
	Now, we begin to prove the Theorem \ref{B}. The idea of this proof is from the proof of  \cite[Proposition 4.1,4.2]{HN} partly.
	\begin{proof}[Proof of Theorem \ref{B}]
	Without loss of generality, we can assume that $f_{1},f_{2},g_{1},\cdots,\\g_{d}$ are bounded real valued functions which take value on interval $[-1,1]$. Let $\Omega=[-1,1]^{\Z}$. Let $\mathcal{F}$ be the Borel $\sigma$-algebra of $\Omega$. The rest proof is divided into two parts.
	
		\textbf{Part I}: We verify the following: If we can find $j$ from $\{1,\cdots,d\}$ such that $g_{j}$ has zero conditional expectation with respect to $\mathcal{Z}_{\infty}(S_{j})$, then $$\lim_{N\rightarrow \infty}\frac{1}{N}\sum_{n=1}^{N}f_{1}(T^{an}x)f_{2}(T^{bn}x)\prod_{j=1}^{d}g_{j}(S_{j}^{p_{j}(n)}x)=0$$ in $L^{2}(\mu)$.
		
		Suppose that the above result fails. Then there exist $\epsilon>0$ and a strictly increasing sequence of positive integers $\{N_{k}\}_{k\ge 1}$ such that for any $k\in \N$, $$\norm {\frac{1}{N_k}\sum_{n=1}^{N_k}f_{1}(T^{an}x)f_{2}(T^{bn}x)\prod_{j=1}^{d}g_{j}(S_{j}^{p_{j}(n)}x)}_{2}\ge \epsilon.$$ By Lemma \ref{lem4}, there exists a subsequence $\{N'_{k}\}_{k\ge 1}$ of $\{N_{k}\}_{k\ge 1}$ such that for $\mu$-a.e $x\in X$, the sequence $\{\prod_{j=1}^{d}g_{j}(S_{j}^{p_{j}(n)}x)\}_{n\in \Z}$ admits a correlation on sequence $\{N'_{k}\}_{k\ge 1}$. Moreover, the $0$-th coordinate projection $F_{0}:\Omega\rightarrow \R$ has zero conditional expectation with respect to the Pinsker factor of the corresponding Furstenberg system $(\Omega, \mathcal{F},\nu_{x},\sigma)$.
		
		By \cite[Proposition 2.5]{HN} and \cite[Theorem 4.13.(ii)]{Wa}, for $\mu$-a.e $x\in X$, the sequence $\{f_{1}(T^{an}x)\}_{n\in \Z}$ admits a correlation on sequence $\{N'_{k}\}_{k\ge 1}$ and the corresponding Furstenberg system $(\Omega, \mathcal{F},\lambda_{x}^{a},\sigma)$ is ergodic and has zero entropy. Let $G_{0}^{a}$ be its $0$-th coordinate projection. Likely, for $\mu$-a.e $x\in X$, the sequence $\{f_{2}(T^{bn}x)\}_{n\in \Z}$ admits a correlation on sequence $\{N'_{k}\}_{k\ge 1}$ and the corresponding Furstenberg system $(\Omega, \mathcal{F},\lambda_{x}^{b},\sigma)$ is ergodic and has zero entropy. Let $G_{0}^{b}$ be its $0$-th coordinate projection.
		
		So we can find a full measure subset $X_0$ of $X$ such that for any $x\in X_0$, the previous properties hold. By our assumption, there exists a subset $X_{1}$ of $X_0$ with $\mu(X_1)>0$ such that for any $x\in X_1$, there exists a subsequence $\{N'_{x,k}\}_{k\ge 1}$ of $\{N'_{k}\}_{k\ge 1}$ such that $$\lim_{k\rightarrow \infty}\frac{1}{N'_{x,k}}\sum_{n=1}^{N'_{x,k}}f_{1}(T^{an}x)f_{2}(T^{bn}x)\prod_{j=1}^{d}g_{j}(S_{j}^{p_{j}(n)}x)$$ exists and is non-zero.
		
		Let $x\in X_1$ be fixed at here. Let $w_{a},w_{b},z\in \Omega$ defined by $w_{a}(n)=f_{1}(T^{an}x),w_{b}(n)\\=f_{2}(T^{bn}x),z(n)=\prod_{j=1}^{d}g_{j}(S_{j}^{p_{j}(n)}x)$. Then there exists a subsequence $\{N''_{x,k}\}_{k\ge 1}$ of $\{N'_{x,k}\}_{k\ge 1}$ such that the sequence $$\Big\{\frac{1}{N''_{x,k}}\sum_{n=1}^{N''_{x,k}}\delta_{(\sigma^{n}w_{a},\sigma^{n}w_{b})}\Big\}_{k\ge 1}$$ has weak-$*$ limit $\lambda_x$ and the sequence  $$\Big\{\frac{1}{N''_{x,k}}\sum_{n=1}^{N''_{x,k}}\delta_{(\sigma^{n}w_{a},\sigma^{n}w_{b},\sigma^{n}z)}\Big\}_{k\ge 1}$$  has weak-$*$ limit $\rho_x$. Clealy, $\rho_x$ is a joining of $(\Omega, \mathcal{F},\nu_{x},\sigma)$ and $(\Omega\times \Omega,\mathcal{F}^{\otimes 2}, \lambda_{x},\sigma\times \sigma)$. By \cite[Fact 4.4.3]{T}, we know that $(\Omega\times \Omega, \mathcal{F}^{\otimes 2},\lambda_{x},\sigma\times \sigma)$ has zero entropy since $\lambda_x$ is a joining of $(\Omega,\mathcal{F}, \lambda_{x}^{a},\sigma)$ and $(\Omega, \mathcal{F},\lambda_{x}^{b},\sigma)$. By \cite[Proposition 2.1]{HN}, we have
		$$\lim_{k\rightarrow \infty}\frac{1}{N''_{x,k}}\sum_{n=1}^{N''_{x,k}}f_{1}(T^{an}x)f_{2}(T^{bn}x)\prod_{j=1}^{d}g_{j}(S_{j}^{p_{j}(n)}x)=\int F_{0}\otimes G_{0}^{a}\otimes G_{0}^{b}d\rho_x=0.$$ It is a contradiction. So $$\lim_{N\rightarrow \infty}\frac{1}{N}\sum_{n=1}^{N}f_{1}(T^{an}x)f_{2}(T^{bn}x)\prod_{j=1}^{d}g_{j}(S_{j}^{p_{j}(n)}x)=0$$ in $L^{2}(\mu)$.
		
			\textbf{Part II}: we verify the following: If for any $1\le j\le d$, $g_{j}$ is measurable with respect to $\mathcal{Z}_{\infty}(S_{j})$, then $$\lim_{N\rightarrow \infty}\frac{1}{N}\sum_{n=1}^{N}f_{1}(T^{an}x)f_{2}(T^{bn}x)\prod_{j=1}^{d}g_{j}(S_{j}^{p_{j}(n)}x)$$ exists almost everywhere.
			
			By \cite[Theorem 14.15,16.10]{HK-book}, we know that for any $1\le j\le d$, we can find a function sequence $\{h_{j,m}\}_{m\ge 1}$ such that for each $m\ge 1$, the followings hold:
			\begin{itemize}
				\item[(1)]$\norm {h_{j,m}-g_{j}}_{d+3}\le \frac{1}{4^{m+1}}$;
				\item[(2)]for $\mu$-a.e. $x\in X$, $\{h_{j,m}(S_{j}^{p_{j}(n)}x)\}_{n\in \Z}$ is a nilsequence.
			\end{itemize}
				Note that product of finite nilsequences is still a nilsequence. By Corollary \ref{cor1}, we know for any $m\ge 1$, $$\lim_{N\rightarrow \infty}\frac{1}{N}\sum_{n=1}^{N}f_{1}(T^{an}x)f_{2}(T^{bn}x)\prod_{j=1}^{d}h_{j,m}(S_{j}^{p_{j}(n)}x)$$ exists almost everywhere.
				By repeating the arguments of the proof of \cite[Corollary 2.2]{DL}, we know that $$\lim_{N\rightarrow \infty}\frac{1}{N}\sum_{n=1}^{N}f_{1}(T^{an}x)f_{2}(T^{bn}x)\prod_{j=1}^{d}g_{j}(S_{j}^{p_{j}(n)}x)$$ exists almost everywhere. This finishes the proof.
	\end{proof}

 \section*{Acknowledgement}
 The author is supported by NNSF of China (11971455, 12031019, 12090012). The author's thanks go to Professor Song Shao and Professor Xiangdong Ye for their useful suggestions.

 \bibliographystyle{plain}
 \bibliography{ref}
 
 \end{document}